\documentclass[a4paper,10pt]{article}
\usepackage{amsmath,amsthm,amssymb}

\numberwithin{equation}{section}

\newcommand{\ot}{\otimes}

\newcommand{\Z}{\mathbb{Z}}
\newcommand{\R}{\mathbb{R}}
\newcommand{\N}{\mathbb{N}}
\newcommand{\C}{\mathbb{C}}

\newcommand{\rE}{\operatorname{E}}

\newcommand{\rL}{\mathord{\text{\rm L}}}
\newcommand{\actson}{\curvearrowright}
\newcommand{\recht}{\rightarrow}
\newcommand{\cT}{\mathcal{T}}

\newcommand{\Stab}{\operatorname{Stab}}

\newcommand{\embed}{\prec}

\newcommand{\al}{\alpha}
\newcommand{\be}{\beta}

\newcommand{\vphi}{\varphi}

\newcommand{\Ptilde}{\widetilde{P}}

\newcommand{\bim}[3]{\mathord{\raisebox{-0.4ex}[0ex][0ex]{\scriptsize $#1$}{#2}\hspace{-0.2ex}\raisebox{-0.4ex}[0ex][0ex]{\scriptsize $#3$}}}

\newcommand{\Gam}{\Gamma}
\newcommand{\LGam}{{\rm L}(\Gamma)}

\newcommand{\LH}{{\rm L}(H)}
\newcommand{\LSigma}{{\rm L}(\Sigma)}

\newcommand{\HNN}{\operatorname{HNN}}

\theoremstyle{plain}
\newtheorem{theorem}{Theorem}[section]

\newtheorem{lemma}[theorem]{Lemma}
\newtheorem{proposition}[theorem]{Proposition}
\newtheorem{corollary}[theorem]{Corollary}
\theoremstyle{definition}

\begin{document}

\begin{center}
{\large \bf A NOTE ON THE VON NEUMANN ALGEBRA OF A BAUMSLAG-SOLITAR GROUP}

\bigskip

{\sc Pierre Fima$^{(1,2)}$
\setcounter{footnote}{1}\footnotetext{Partially supported by ERC Starting Grant VNALG-200749.}
\setcounter{footnote}{2}\footnotetext{Institut Math\'ematiques de Jussieu, Universit\'e Paris Diderot, 175 rue de Chevaleret, 75013 Paris.
    \\ E-mail: pfima@math.jussieu.fr}}
\end{center}

\begin{abstract}
\noindent We study qualitative properties of the group von Neumann algebra of a Baumslag-Solitar group. Namely, we prove that, in the non-amenable and {ICC} case, the associated ${\rm II}_1$ factor is prime, not solid, and does not have any Cartan subalgebra.
\end{abstract}

\section{Introduction}

In their breakthrough paper \cite{OP07} Ozawa and Popa obtained new structural results for group (and group measure space) von Neumann algebras. In particular, they showed that ${\rm L}(\mathbb{F}_n)$ is \textit{strongly solid}, i.e. for any diffuse amenable von Neumann subalgebra $Q\subset {\rm L}(\mathbb{F}_n)$, the normalizer of $Q$ generates an amenable von Neumann subalgebra. This strengthened two well-known results: Voiculescu's result \cite{Vo96} showing that ${\rm L}(\mathbb{F}_n)$ has no Cartan subalgebra and Ozawa's result \cite{Oz04} showing that ${\rm L}(\mathbb{F}_n)$ is \textit{solid} (the relative commutant of any diffuse von Neumann subalgebra is amenable) which itself strengthened Ge's result \cite{Ge98} showing that  ${\rm L}(\mathbb{F}_n)$ is \textit{prime} (it is not decomposable into tensor product of ${\rm II}_1$ factors).

In this note we study these properties for Baumslag-Solitar groups factors.

Let $n,m\in\Z-\{0\}$. The Baumslag-Solitar group is defined by ${\rm BS}(n,m):=\langle a,b|\, ab^na^{-1}=b^m\rangle$. It was proved in \cite{Mo91} that ${\rm BS}(n,m)\simeq {\rm BS}(p,q)$ if and only if $\{n,m\}=\{\epsilon p,\epsilon q\}$ for some $\epsilon\in\{-1,1\}$. Moreover, $\Gam$ is known to be non-amenable but inner-amenable and ICC whenever $|n|,|m|\geq 2$ and $|n|\neq |m|$ (see \cite{St06}). Gal and Januszkiewicz \cite{GJ03} proved that ${\rm BS}(n,m)$ has the Haagerup property. Note that their proof also implies that ${\rm BS}(n,m)$ has the complete approximation property (CMAP). Actually one just has to check that the automorphism group of a locally finite tree has the CMAP as a locally compact group (for the compact-open topology). Also, the first ${\rm L}^2$ Betti number of ${\rm BS}(n,m)$ is zero.

 Our results can be summarized as follows.

\begin{theorem}\label{main}
Let $\Gam={\rm BS}(n,m)$. Assume $|n|,|m|\geq 2$ and $|n|\neq |m|$. The following holds.
\begin{enumerate}
\item $\LGam$ is a prime $\rm{II}_1$ factor.
\item $\LGam$ is not solid and does not have any Cartan subalgebra.
\end{enumerate}
\end{theorem}

We prove actually a general primeness result for groups acting on trees (see Corollary \ref{primetree}). We also prove a stronger property than the absence of Cartan subalgebra (see Theorem \ref{anti-solid}). Namely, we prove that ${\rm L}(\Gamma)$ is \textit{robust} i.e, the relative commutant of any regular and amenable von Neumann subalgebra is non-amenable.

\section{Preliminaries}

\subsection{Weakly compact actions}

Weakly compact actions were introduced by Ozawa-Popa \cite{OP07}. The following theorem is similar to Theorem 4.9 in \cite{OP07}. The main ingredients of the proof are contained in the proofs of Theorem 4.9 in \cite{OP07} and Theorem B in \cite{OP08} as explained in the proof of Theorem 3.5 in \cite{HS09} (see also Theorem 3.3 in \cite{Ho10}). This result is not stated explicitly in any of these papers but the proof is the same as the one of Theorem 3.5 in \cite{HS09}.

\begin{theorem}\label{OP2}
Let $P$ be a tracial von Nemann algebra that admits the following deformation property: there exists a tracial von Neumann algebra $\Ptilde$, a trace preserving inclusion $P\subset \Ptilde$ and a one-parameter group $(\alpha_s)_{s\in\R}$ of trace-preserving automorphisms of $\Ptilde$ such that 
\begin{itemize}
\item ${\rm lim}_{s\rightarrow 0}||\alpha_s(x)-x||_2=0$ for all $x\in P$.
\item $\bim{P}{{\rm L}^2(\Ptilde)\ominus{\rm L}^2(P)}{P}$ is weakly contained in  $\bim{P}{{\rm L}^2(P)\ot{\rm L}^2(P)}{P}$.
\item There exists $c>0$ such that $||\alpha_{2s}(x)-x||_2\leq c||\alpha_s(x)-E_P\circ\alpha_s(x)||_2$, for all $x\in P,s\in\R$ (transversality).
\end{itemize}
Let $Q\subset P$ be a von Neumann subalgebra and $G\subset\mathcal{N}_{P}(Q)$ be a subgroup such that the action $G\curvearrowright Q$ is weakly compact. If, for all non zero projection $z$ in $\mathcal{Z}(G'\cap P)$, $\alpha_s$ does not converge uniformly on $(zQ)_1$, then $G''$ is amenable.
\end{theorem}

\subsection{{\rm HNN} extensions of von Neumann algebras}\label{HNNVN}

{\rm HNN} extensions of general von Neumann algebras were introduced in \cite{Ue04}. In this note we follow the approach of \cite{FV10}. Let $(M,\tau_M)$ be a tracial von Neumann algebra and $N\subset M$ a von Neumann subalgebra. Let $\theta\,:\,N\rightarrow M$ be trace-preserving embedding. Let $P=\HNN(M,N,\theta)$ be the HNN extension. We recall that a tracial von Neumann algebra $\Ptilde$ with a trace preserving inclusion $P\subset\Ptilde$ and one parameter group of automorphisms $(\alpha_s)$ of $\Ptilde$ satisfying the first condition of Theorem \ref{OP2} were constructed in \cite{FV10}. It was also observed that $\Ptilde = P *_N (N \ot \rL \Z)$. This implies that, when $N$ is amenable, $\bim{P}{{\rm L}^2(\Ptilde) \ominus {\rm L}^2(P)}{P}$ is weakly contained in the coarse $P$-$P$-bimodule. A detailed argument can be found e.g.\ in \cite[Proposition 3.1]{CH08}. Also, an automorphism $\beta\in\text{Aut}(P)$ such that $\beta \circ \al_s = \al_{-s} \circ \be$ and $\beta(x) = x$ for all $x \in P$ was introduced in \cite{FV10}. Such a deformation is \textit{s-malleable}. As such, it automatically satisfies the following transversality property (see Lemma 2.1 in \cite{Po08}):
$$||\alpha_{2s}(x)-x||_2\leq 2||\alpha_s(x)-E_M\circ\alpha_s(x)||_2\quad\text{for all}\quad x\in M,s\in\R.$$
Hence, if $N$ is amenable, Theorem \ref{OP2} applies to the {\rm HNN} extension $P=\HNN(M,N,\theta)$.

\section{Primeness results for groups acting on trees}

We will use the following proposition. The proof is similar to the one of Theorem 5.2 in \cite{CH08} (even easier because we state it in the finite case).

\begin{proposition}
Let $M_1$ and $M_2$ be finite von Neumann algebras with a common von Neumann subalgebra $B$ of type ${\rm I}$. Let $M=M_1*_BM_2$. Let $p\in M$ be a non zero projection. If $pMp$ is a non amenable ${\rm II}_1$ factor then $pMp$ is prime.
\end{proposition}

The following result is a direct corollary of the preceding Proposition and Remark 4.6 in \cite{FV10}. 

\begin{corollary}\label{primeHNN}
Let $\Gam=\HNN(H,\Sigma,\theta)$ be a non-trivial $\HNN$ extension (i.e. $\Sigma,\theta(\Sigma)\neq H$). Assume that the following conditions are satisfied.
\begin{enumerate}
\item $\Gam$ is non-amenable and {\rm ICC}.
\item $\Sigma$ is abelian or finite.
\end{enumerate}
Then, $\LGam$ is a prime $\rm{II}_1$ factor.
\end{corollary}

We now extend this result to groups acting on trees.

\begin{corollary}\label{primetree}
Let $\Gam$ by a group satisfying the following properties.
\begin{itemize}
\item $\Gam$ is non-amenable and {\rm ICC}.
\item $\Gamma$ admits an action $\Gamma \actson \cT$ without inversion on a tree $\cT$ such that there exists a finite subtree with a finite stabilizer and such that there exists an edge $e \in \rE(\cT)$ with the properties that $\Stab e$ is abelian or finite and that the smallest subtrees containing all vertices $\Gamma \cdot s(e)$, resp.\ $\Gamma \cdot r(e)$, are both equal to the whole of $\cT$.
\end{itemize}
Then ${\rm L}(\Gam)$ is a prime $\rm{II}_1$ factor.
\end{corollary}

\begin{proof}
In Corollary \ref{primeHNN} we have seen that the conclusion of the corollary holds for certain HNN extensions. In \cite[Theorem 5.2]{CH08} it was shown that the conclusion also holds for certain amalgamated free product groups. So, it suffices to prove that all groups satisfying the assumptions of the corollary fall into one of both families and this can be done as in the proof of Theorem 1.2 in \cite{FV10}.
\end{proof}

\section{Robustness for certain {\rm HNN} extensions}

We call a von Neumann algebra \textit{robust} if the relative commutant of any regular and amenable von Neumann subalgebra is non-amenable. Clearly, robustness implies the absence of Cartan subalgebra.

Let $(P,\tau)$ be a tracial von Neumann algebra and $A, B\subset M$ be possibly non-unital von Neumann subalgebras. Following \cite[Section 2]{Po03} we say that $A$ \textit{embeds into} $B$ inside $M$, and we denote it by $A \embed_M B$, if there exist $n\in\N$, a non-zero projection $q \in B^n=M_n(\C)\otimes B$, a normal unital $*$-homomorphism $\vphi : A \recht qB^nq$ and a non-zero partial isometry $v \in 1_A(M_{1,n}(\C)\ot M)$ satisfying $a v = v \vphi(a)$ for all $a \in A$. Otherwise, we write $A \not\embed_M B$.

\begin{lemma}\label{intertwining}
Let $P=\HNN(M,N,\theta)$ be an $\HNN$ extension of finite von Neumann algebras and suppose that $N$ is amenable and $P$ has the CMAP. Let $Q\subset P$ a unital von Neumann subalgebra. If $Q$ is amenable and $Q\underset{P}{\nprec} M$ then $\mathcal{N}_P(Q)''$ is amenable.
\end{lemma}

\begin{proof}
Let $z\in\mathcal{N}_{P}(Q)'\cap P$ be a non-zero projection. Observe that $z\in Q'\cap P$. As $Q\underset{P}{\nprec} M$ we get $zQ\underset{P}{\nprec} M$. By \cite[Theorem 3.4]{FV10} we get that the deformation $(\alpha_s)$ does not converge uniformly on the unit ball of $zQ$, for all non-zero projection $z\in\mathcal{N}_{P}(Q)'\cap P$. We can apply \cite[Theorem 3.5]{OP07} and  Theorem \ref{OP2} to conclude that $\mathcal{N}_{P}(Q)''$ is amenable.
\end{proof}

\begin{theorem}\label{anti-solid}
Let $\Gam=\HNN(H,\Sigma,\theta)$. Suppose that the following conditions are satisfied.
\begin{enumerate}
\item $H$ is abelian, $2\leq|H/\Sigma|<\infty$ and $3\leq|H/\theta(\Sigma)|$,
\item $\Gam$ has the CMAP.
\end{enumerate}
Then $\LGam$ is robust. If moreover $\Sigma$ is infinite, then $\LGam$ is not solid.
\end{theorem} 

\begin{proof}
Let $\Gamma=\langle H,t\,|\theta(\sigma)=t\sigma t^{-1}\,\forall\sigma\in\Sigma\rangle$. Define $G=\langle H, t^{-1}Ht\rangle\subset\Gam$ and $\Sigma'=\{g\in\Gam\,|\,g\sigma=\sigma g\,\,\text{for all}\,\,\sigma\in\Sigma\}$. As $H$ is abelian, we have $H\subset\Sigma'$. Moreover, for all $\sigma\in\Sigma$ and $h\in H$, we have
$$t^{-1}ht\sigma=t^{-1}h\theta(\sigma)t=t^{-1}\theta(\sigma)ht=\sigma t^{-1}ht.$$
It follows that $t^{-1}Ht\subset \Sigma'$. We conclude that $G\subset\Sigma'$.

Let $\widetilde{H}$ be a copy of $H$ and view $\Sigma$ as a subgroup of $\widetilde{H}$ via the map $\theta$. Define the following group homomorphisms: the first one from $H$ to $G$ is the identity, the second one from $\widetilde{H}$ to $G$ maps $h$ onto $t^{-1}ht$. These groups homomorphisms agree on $\Sigma$ (because we see $\Sigma<\widetilde{H}$ via the map $\theta$). We get a group homomorphism from $H\underset{\Sigma}{*}\widetilde{H}$ to $G$ which is clearly surjective. It is also injective because it maps each reduced word in $H\underset{\Sigma}{*}\widetilde{H}$ onto a reduced word (in the HNN extension sense) in $G$. As $|H/\Sigma|\geq 2$ and $|H/\theta(\Sigma)|\geq 3$, $H\underset{\Sigma}{*}\widetilde{H}$ is not amenable.

Let $Q\subset \LGam$ be an amenable regular subalgebra. By Lemma \ref{intertwining}, $Q\underset{\LGam}\prec \LH$. Because $\Sigma$ has finite index in $H$ we obtain $Q\underset{\LGam}\prec \LSigma$. It follows that $\LSigma'\cap\LGam \underset{\LGam}\prec Q'\cap M$. In particular, ${\rm L}(G)\underset{\LGam}\prec Q'\cap M$. Because ${\rm L}(G)$ has no amenable direct summand, $ Q'\cap M$ is not amenable. If $\LSigma$ is infinite, $\LGam$ is obviously not solid. Actually, $\LSigma$ is a diffuse amenable von Neumann subalgebra and its relative commutant is not amenable as it contains ${\rm L}(G)$.
\end{proof}

We obtain the following obvious corollary.

\begin{corollary}
Let $\Gam={\rm BS}(m,n)$. If $|n|,|m|\geq 2$ and $|n|\neq |m|$ then $\LGam$ is a non-solid $\rm{II}_1$ factor and does not have any Cartan subalgebra.
\end{corollary}


\section*{Acknowledgements} Many thanks are due to Stefaan Vaes for several helpful conversations.


\end{document}